\documentclass{article}
\usepackage{geometry}
\usepackage{fancyhdr}
\usepackage{amsmath,amsthm,amssymb}
\usepackage{graphicx}
\usepackage{hyperref}
\usepackage{tikz-cd}
\usepackage{mathabx}
\usepackage[T1]{fontenc}
\usepackage[utf8]{inputenc}
\usepackage[english]{babel}
\usepackage{tcolorbox}
\usepackage{float}
\usepackage{caption}

\usepackage{lmodern}

\theoremstyle{definition}
\newtheorem{theo}{Theorem}[section]

\theoremstyle{definition}
\newtheorem{defin}{Definition}[section]

\theoremstyle{definition}
\newtheorem{conjecture}{Conjecture}[section]

\theoremstyle{definition}
\newtheorem{lemma}{Lemma}[section]

\DeclareMathOperator{\Li}{Li}

\catcode`,\active

\catcode`\,12

\catcode`,\active

\catcode`\,12

\usepackage{verbatim}
\immediate\write18{texcount -tex -sum  \jobname.tex > \jobname.wordcount.tex}

\title{Linear relations between logarithmic integrals of high weight and some closed-form evaluations}
\author{Kam Cheong Au}
\date{} 

\begin{document}
\maketitle
\begin{abstract}
The focus of our investigation will be integrals of form $\int_0^1 \log^a(1-x) \log^b x \log^c(1+x) /f(x) dx$, where $f$ can be either $x,1-x$ or $1+x$. We show that these integrals possess a plethora of linear relations, and give systematic methods of finding them. In lower weight cases, these linear relations yields remarkable closed-forms of individual integrals; in higher weight, we discuss the $\mathbb{Q}$-dimension that such integrals span. Our approach will not feature Euler sums. 
\end{abstract} 

\hspace{10pt}

\section{Introduction}
Throughout, we will denote
$$\begin{aligned}i_{abc0} &= \int_0^1 \frac{\log^a(1-x) \log^b x \log^c(1+x)}{1-x} dx \\
i_{abc1} &= \int_0^1 \frac{\log^a(1-x) \log^b x \log^c(1+x)}{x} dx \\
i_{abc2} &= \int_0^1 \frac{\log^a(1-x) \log^b x \log^c(1+x)}{1+x} dx
\end{aligned}$$
The weight associated to any of them is $w=1+a+b+c$. For convergence reason, we need to assume $b\geq 1$ in $i_{abc0}$, $a\geq 1$ or $c\geq 1$ in $i_{abc1}$. They (non-convergent ones excluded) are collective called \textbf{polylogarithm integrals of weight $w$}.

\begin{lemma}
There are $(3w^2+w-2)/2$ polylogarithm integrals of weight $w$.
\end{lemma}
\begin{proof}
An easy counting exercise. For example, integrals of weight $2$ are: $i_{0100}, i_{0011}, i_{1001}, i_{0012}, i_{0102}, i_{1002}$. 
\end{proof}

Our main objective is to calculate these integrals in terms of polylogarithm $\Li_n(1/2)$ and Riemann-zeta function, where
$$\Li_n(z) = \sum_{k=1}^\infty \frac{z^k}{k^n}$$
One remarkable result in this paper is theorem \ref{remarkableweight8}:
\begin{multline*}
\int_0^1 \frac{\log^2 (1-x) \log^2 x \log^3(1+x)}{x} dx = -168 \text{Li}_5\left(\frac{1}{2}\right) \zeta (3)+96 \text{Li}_4\left(\frac{1}{2}\right){}^2-\frac{19}{15} \pi ^4 \text{Li}_4\left(\frac{1}{2}\right)+\\ 12 \pi ^2 \text{Li}_6\left(\frac{1}{2}\right)+8 \text{Li}_4\left(\frac{1}{2}\right) \log ^4(2)-2 \pi ^2 \text{Li}_4\left(\frac{1}{2}\right) \log ^2(2)+12 \pi ^2 \text{Li}_5\left(\frac{1}{2}\right) \log (2)+\frac{87 \pi ^2 \zeta (3)^2}{16}+\\ \frac{447 \zeta (3) \zeta (5)}{16}+\frac{7}{5} \zeta (3) \log ^5(2)-\frac{7}{12} \pi ^2 \zeta (3) \log ^3(2)-\frac{133}{120} \pi ^4 \zeta (3) \log (2)-\frac{\pi ^8}{9600}+\frac{\log ^8(2)}{6}- \\ \frac{1}{6} \pi ^2 \log ^6(2)-\frac{1}{90} \pi ^4 \log ^4(2)+\frac{19}{360} \pi ^6 \log ^2(2)
\end{multline*}

If explicit result is not available, we seek for linear relations satisfied between them. 


Let $\mathcal{A}$ be the algebra generated over $\mathbb{Q}$ by all polylogarithm integrals, and $\mathcal{I}_n$ be the ideal of $\mathcal{A}$ generated by these integrals which have weight $\leq n$. 
\begin{defin}
Let $i_1,\cdots,i_n$ be some polylogarithm integrals of weight $w$, a \textbf{relation} between them is a non-trivial $\mathbb{Q}$-linear combination of $i_1,\cdots,i_n$ that equals $0$ in $\mathcal{A}/\mathcal{I}_{w-1}$.
\end{defin}
For example, we have $i_{2231} = 0$ in $\mathcal{A}/\mathcal{I}_{7}$.

\par 
Given the well-known conjecture on algebraic independence of multiple zeta values of different weights, it would be reasonable to assume same is true for these integrals, we shall assume this is the case throughout. This implies $\mathcal{A}$ has a natural graded structure, so we can talk about the weight of constants such as $\pi^2$, $\zeta(n)$ and $\Li_n(1/2)$. However, in no way the explicit evaluation or validity of relations in the paper are contingent on this inaccessible assumption.\\ 

It must be mentioned that such integrals are usually attacked in literature via Euler sums like in \cite{1},\cite{2} and \cite{3}: each integral are converted in a combination of Euler sums with known values. However, virtually all discussions available are with weight $\leq 5$. Nonetheless, explicit results such as \ref{remarkableweight8} should also be provable by first converting to alternating MZVs (multiple zeta values), then using various shuffle relations. 

\section{Methods to obtain relations}
\subsection{Integration by parts}
The most straightforward way to obtain relations between $i_{abcd}$, where $d=0,1,2$, is integration by parts. Temporarily write $i_{abcd}$ as $f(a,b,c;d)$, then we have
$$\begin{aligned}f(a,b,c;1 - x) =  - af(a,b,c;1 - x) + bf(a + 1,b - 1,c;x) + cf(a + 1,b,c - 1;1 + x)\\
f(a,b,c;x) = af(a - 1,b + 1,c;1 - x) - bf(a,b,c;x) - cf(a,b + 1,c - 1;1 + x)\\
f(a,b,c;1 + x) = af(a - 1,b,c + 1;1 - x) - bf(a,b - 1,c + 1;x) - cf(a,b,c;1 + x)
\end{aligned}$$
A relation is discarded if any of four integrals on the same horizontal line is non-convergent. 
\\
Note that the elementary $$i_{00n2} = (\log 2)^{n+1} /(n+1)$$ also comes under this method.
~\\[0.01in]

\subsection{Fractional transformation}
We make the substitution $x = \frac{1-u}{1+u}$, then
$$u = \frac{1-x}{1+x} \qquad dx = \frac{-2}{(1+u)^2}du$$
$$\begin{gathered}
  {i_{abc0}} = \int_0^1 {\frac{{{{\log }^a}(\frac{{2u}}{{1 + u}}){{\log }^b}(\frac{{1 - u}}{{1 + u}}){{\log }^c}(\frac{2}{{1 + u}})}}{{u(1 + u)}}du}   \\
  {i_{abc1}} = \int_0^1 {\frac{{{{\log }^a}(\frac{{2u}}{{1 + u}}){{\log }^b}(\frac{{1 - u}}{{1 + u}}){{\log }^c}(\frac{2}{{1 + u}})}}{{(1 - u)(1 + u)}}du}   \\
  {i_{abc2}} = \int_0^1 {\frac{{{{\log }^a}(\frac{{2u}}{{1 + u}}){{\log }^b}(\frac{{1 - u}}{{1 + u}}){{\log }^c}(\frac{2}{{1 + u}})}}{{1 + u}}du} 
\end{gathered} $$
When we perform this substitution on $i_{abc2}$, using
$$\log (\frac{{2u}}{{1 + u}}) = \log 2 - \log u - \log (1 + u)\quad \log (\frac{{1 - u}}{{1 + u}}) = \log (1 - u) - \log (1 + u)\quad \log (\frac{2}{{1 + u}}) = \log 2 - \log (1 + u)$$
we can split RHS integral into a combination of different $i_{a'b'c'2}$, all of them have weight $\leq 1+a+b+c$. Assuming integrals of lower weights have been calculated, we now yield a relation between $i_{abc2}$ and a combinations of different $i_{a'b'c'2}$. For example, performing this substitution on $i_{0212}$, because
\begin{multline*}{\log ^2}(\frac{{1 - u}}{{1 + u}})\log (\frac{2}{{1 + u}}) = \log 2{\log ^2}(1 - u) - 2\log 2\log (1 - u)\log (1 + u) \\ - {\log ^2}(1 - u)\log (1 + u) + \log 2{\log ^2}(1 + u) + 2\log (1 - u){\log ^2}(1 + u) - {\log ^3}(1 + u)
\end{multline*}
therefore
$${i_{0212}} =  - {i_{0032}} + 2{i_{1022}} - {i_{2012}} + (\log 2){i_{0022}} - 2(\log 2){i_{1012}} + (\log 2){i_{2002}}$$
subsisting known values of lower weight ones $i_{0022}, i_{1012}, i_{2002}$ yield a relation of weight $4$ polylogarithm integrals. \par 

The same procedure can be carried out on $i_{abc0}$, The denominator $1/u(1+u) = 1/u - 1/(1+u)$. To ensure every term after expansion is a convergent integral, we need to assume $b\geq 1$. In this way, we can write $i_{abc0}$ into a combination of different $i_{a'b'c'1}$ and $i_{d'e'f'2}$. \par 

Due to convergent issue, we do not perform this substitution on $i_{abc1}$. 
~\\[0.01in]

\subsection{Explicit polylogarithm integrals}
Results from this and next subsection are well-known, but we still choose to record them here so that our methods of obtaining relations are as systematized as possible. This section contains two explicit evaluations, which serves as two relations.
\begin{lemma}
$$\int_0^1 {\frac{{{{\log }^n}(1 - x)}}{{1 + x}}dx}  = {( - 1)^n}n!{\operatorname{Li} _{n + 1}}(\frac{1}{2})$$
\end{lemma}
\begin{proof}
Write
$$\int_0^1 {\frac{{{{\log }^n}(1 - x)}}{{1 + x}}dx}  = \frac{1}{2}\int_0^1 {\frac{{{{\log }^n}x}}{{1 - x/2}}dx} $$
then expand $1/(1-x/2)$ as a geometric series, interchange order of integral and summation (easily justified), immediately gives the result.
\end{proof}

\begin{lemma}
$$\int_0^1 \frac{\log^{n-1} (1+x)}{x} dx = \frac{{{{\log }^n}2}}{n} + (n - 1)!\zeta (n) - \sum\limits_{k = 0}^{n - 1} {\binom{n-1}{k}k!{{(\log 2)}^{n - k - 1}}{{\operatorname{Li} }_{k + 1}}(\frac{1}{2})} $$
\end{lemma}
\begin{proof}
$$\begin{aligned}
  \int_0^1 {\frac{{{{\log }^{n - 1}}(1 + x)}}{x}dx}  &= \int_1^2 {\frac{{{{\log }^{n - 1}}x}}{{x - 1}}dx}  = {( - 1)^{n - 1}}\int_{1/2}^1 {\frac{{{{\log }^{n - 1}}x}}{{x(1 - x)}}dx}  \\
   &= {( - 1)^{n - 1}}\int_{1/2}^1 {\frac{{{{\log }^{n - 1}}x}}{x}dx}  + {( - 1)^{n - 1}}\int_{1/2}^1 {\frac{{{{\log }^{n - 1}}x}}{{1 - x}}dx}  \hfill \\
   &= \frac{{{{\log }^n}2}}{n} + {( - 1)^{n - 1}}\int_{1/2}^1 {\frac{{{{\log }^{n - 1}}x}}{{1 - x}}dx}   \\
   &= \frac{{{{\log }^n}2}}{n} + {( - 1)^{n - 1}}\int_0^1 {\frac{{{{\log }^{n - 1}}x}}{{1 - x}}dx}  - {( - 1)^{n - 1}}\int_0^{1/2} {\frac{{{{\log }^{n - 1}}x}}{{1 - x}}dx}   
\end{aligned}$$
where we make the substitution $x\mapsto 1/x$ in the third step. The second last integral is just $(-1)^{n-1} (n-1)! \zeta(n)$, and 
$$\int_0^{1/2} {\frac{{{{\log }^{n - 1}}x}}{{x - 1}}dx}  = \frac{1}{2}\int_0^1 {\frac{{{{\log }^{n - 1}}(x/2)}}{{(x/2) - 1}}dx}  = \frac{1}{2}\sum\limits_{k = 0}^{n - 1} {\binom{n-1}{k}{{( - \log 2)}^{n - k - 1}}\int_0^1 {\frac{{{{\log }^k}x}}{{x/2 - 1}}dx} } $$
The integral in the sum can be easily calculated using geometric series.
\end{proof}
~\\[0.01in]

\subsection{Gamma function}
Recall Euler's beta function, when $\Re(a),\Re(b)>0$, 
\begin{equation}\int_0^1 {{x^{a - 1}}{{(1 - x)}^{b - 1}}dx}  = \frac{{\Gamma (a)\Gamma (b)}}{{\Gamma (a + b)}}\end{equation}
Now expand both sides into series around $a=0,b=1$, the coefficient of $a^n (b-1)^m$ of LHS is (assuming $m\geq 1$)
$$\frac{1}{{n!m!}}\int_0^1 {\frac{{{{\log }^n}x{{\log }^m}(1 - x)}}{x}dx} $$
The RHS can be written as
\begin{equation}\frac{1}{a}\frac{{\Gamma (a + 1)\Gamma (b)}}{{\Gamma (a + b)}} = \frac{1}{a}\exp \left[ {\sum\limits_{k = 2}^\infty  {\frac{{{{( - 1)}^k}\zeta (k)}}{k}{a^k}}  + \sum\limits_{k = 2}^\infty  {\frac{{{{( - 1)}^k}\zeta (k)}}{k}{{(b - 1)}^k}}  - \sum\limits_{k = 2}^\infty  {\frac{{{{( - 1)}^k}\zeta (k)}}{k}{{(a + b - 1)}^k}} } \right]\end{equation}
where we used the expansion
$$\log \Gamma (1 + z) =  - \gamma z + \sum\limits_{k = 2}^\infty  {\frac{{{{( - 1)}^k}\zeta (k)}}{k}{z^k}} $$
valid for $|z|<1$, with $\gamma$ the Euler-Mascheroni constant. Using the usual expansion of $\exp z$, the coefficient of $a^n (b-1)^m$ in $(2)$ can be figured out using finitely many calculations. Therefore we proved

\begin{lemma}\label{Gamma1}
The value of $i_{nm01}$ and $i_{nm00}$ are in the algebra over $\mathbb{Q}$ generated by $\zeta(n)$ and $\pi^2$, with $n=3,5,7,\cdots$.
\end{lemma}

Now consider the expansion of $(1)$ around $a=1/2, b=0$, the coefficient of $a^n (b-1)^m$ of LHS is (assuming $n\geq 1$):
$$\frac{1}{n!m!}\int_0^1 {{x^{ - 1/2}}\frac{{{{\log }^n}x{{\log }^m}(1 - x)}}{{1 - x}}dx} $$
Using the duplication formula $\Gamma(z)\Gamma(z+1/2) = 2^{1-2z}\sqrt{\pi}\Gamma(2z)$, one readily obtain the expansion of $\log\Gamma(z)$ at $z=1/2$:
$$\log \Gamma (z) = \log \Gamma (z) = \frac{{\log \pi }}{2} - \gamma (z - \frac{1}{2}) + (1 - 2z)\log 2 + \sum\limits_{k = 2}^\infty  {\frac{{{{( - 1)}^k}\zeta (k)({2^k} - 1)}}{k}{{(z - \frac{1}{2})}^k}} $$
Therefore the RHS of $(1)$ becomes
$$\frac{1}{b}\exp \left( {2b\log 2 + \sum\limits_{k = 2}^\infty  {\frac{{{{( - 1)}^k}\zeta (k)}}{k}{b^k}}  + \sum\limits_{k = 2}^\infty  {\frac{{{{( - 1)}^k}\zeta (k)({2^k} - 1)}}{k}{{(a - \frac{1}{2})}^k}}  - \sum\limits_{k = 2}^\infty  {\frac{{{{( - 1)}^k}\zeta (k)({2^k} - 1)}}{k}{{(a + b - \frac{1}{2})}^k}} } \right)$$
the coefficient of $a^n (b-1)^m$ in $(2)$ can be figured out using finitely many calculations. Therefore 
\begin{lemma}\label{Gamma2}
The value of $$\int_0^1 {{x^{ - 1/2}}\frac{{{{\log }^n}x{{\log }^m}(1 - x)}}{{1 - x}}dx} $$ when $n\geq 1$ is in the algebra over $\mathbb{Q}$ generated by $\log 2$, $\pi^2$ and $\zeta(n)$, $n=3,5,7,\cdots$.
\end{lemma}
~\\[0.01in]

\subsection{Square replacement}
We begin with an example,
$$\int_0^1 {\frac{{{{\log }^3}(1 + x)}}{x}dx}  + 3\int_0^1 {\frac{{{{\log }^2}(1 + x)\log (1 - x)}}{x}dx}  + 3\int_0^1 {\frac{{\log (1 + x){{\log }^2}(1 - x)}}{x}dx}  + \int_0^1 {\frac{{{{\log }^3}(1 - x)}}{x}dx}  = \int_0^1 {\frac{{{{\log }^3}(1 - {x^2})}}{x}dx} $$
making substitution $u=x^2$ in the last integral gives the relation
$${i_{0031}} + 3{i_{1021}} + 3{i_{2011}} + {i_{3001}} = \frac{1}{2}{i_{3001}}$$

Similar process can be carried out on $i_{ab01}$, with
$$\int_0^1 {\frac{{{{\log }^b}x{{\left[ {\log (1 + x) + \log (1 - x)} \right]}^a}}}{x}dx}  = \frac{1}{{{2^{b + 1}}}}\int_0^1 {\frac{{{{\log }^b}x{{\log }^a}(1 - x)}}{x}dx}$$
Expand the LHS, it involves multiple $i_{a'bc'1}$, RHS can be calculated using $\zeta(n)$ by Lemma . For weight $w$ polylogarithm integral, $b$ can ranges from $0$ to $w-2$, giving $w-1$ relations.\\

There is a second procedure, consider
$$\int_0^1 {\frac{{{{\log }^b}x{{\left( {\log (1 - x) + \log (1 + x)} \right)}^a}}}{{1 - x}}dx}  = \frac{1}{{{2^{b + 1}}}}\int_0^1 {{x^{ - 1/2}}\frac{{{{\log }^b}x{{\log }^a}(1 - x)}}{{1 - {x^{1/2}}}}dx} $$
Using $\frac{{{x^{ - 1/2}}}}{{1 - {x^{1/2}}}} = \frac{1}{{1 - x}} + \frac{{{x^{ - 1/2}}}}{{1 - x}}$, we have
$$\int_0^1 {\frac{{{{\log }^b}x{{\left( {\log (1 - x) + \log (1 + x)} \right)}^a}}}{{1 - x}}dx}  = \frac{1}{{{2^{b + 1}}}}\int_0^1 {\frac{{{{\log }^b}x{{\log }^a}(1 - x)}}{{1 - x}}dx}  + \frac{1}{{{2^{b + 1}}}}\int_0^1 {{x^{ - 1/2}}\frac{{{{\log }^b}x{{\log }^a}(1 - x)}}{{1 - x}}dx}$$
Both integrals on the right can be calculated via Lemmas \ref{Gamma1} and \ref{Gamma2}. Expanding the left side again yield a relation, $b$ here can ranges from $1$ to $w-1$, again giving $w-1$ relations. \par 
Empirically, this method contributes $2w-4$ relations for large weight $w$. 
~\\[0.01in]

\subsection{Contour integration}\label{CIsec}
The following scheme of obtaining relations starts to exert its power when weight $\geq 6$. Let $\log z$ denote the principal branch logarithm, assuming $q,r\geq 1$, we integrate
$$\frac{{{{\log }^p}(1 + z){{\log }^q}(1 - z){{\log }^r}z}}{{z(1 - z)}}$$
around a large semicircle with diameter lying on real axis. The integral along semicircle tends to $0$. Integrals along real axis can be broken piece-wisely into:
$$\begin{aligned}
&\text{Along } [0,1]: \int_0^1 {\frac{{{{\log }^p}(1 + x){{\log }^q}(1 - x){{\log }^r}x}}{{x(1 - x)}}dx} \\
&\text{Along } [1,\infty): \int_1^\infty  {\frac{{{{\log }^p}(1 + x){{\left[ {\log (x - 1) - \pi i} \right]}^q}{{\log }^r}x}}{{x(1 - x)}}dx} \\& \qquad\qquad \qquad =  - \int_0^1 {\frac{{{{\left[ {\log (1 + x) - \log x} \right]}^p}{{\left[ {\log (1 - x) - \log x - \pi i} \right]}^q}{{( - 1)}^r}{{\log }^r}x}}{{1 - x}}dx}  \\
&\text{Along } [-1,0]: \int_{ - 1}^0 {\frac{{{{\log }^p}(1 + x){{\log }^q}(1 - x){{\left[ {\log ( - x) + \pi i} \right]}^r}}}{{x(1 - x)}}dx}  =  - \int_0^1 {\frac{{{{\log }^p}(1 - x){{\log }^q}(1 + x){{\left[ {\log x + \pi i} \right]}^r}}}{{x(1 + x)}}dx} \\
&\text{Along } (-\infty,-1]: \int_{ - \infty }^{ - 1} {\frac{{{{\left[ {\log ( - 1 - x) + \pi i} \right]}^p}{{\log }^q}(1 - x){{\left[ {\log ( - x) + \pi i} \right]}^r}}}{{x(1 - x)}}dx} \\& \qquad\qquad \qquad =  - \int_0^1 {\frac{{{{\left[ {\log (1 - x) - \log x + \pi i} \right]}^p}{{\left[ {\log (1 + x) - \log x} \right]}^q}{{\left[ { - \log x + \pi i} \right]}^r}}}{{1 + x}}dx} 
\end{aligned}$$

The four integrals add up to $0$. When $p+q+r = w-1$, comparing the \textit{real part} gives a relation between weight $w$ polylogarithmic integrals; when $p+q+r = w$, comparing \textit{imaginary part} also gives a relation between weight $w$ polylogarithmic integrals. \\
Note that the relation we obtained is of form
$$\text{sum of weight }w \text{ integrals} = \sum \pi^{2i} (\text{sum of weight }w-2i \text{ integrals})$$ 

When $w=6$, the constant $\pi^2 \Li_4(1/2)$, which cannot be generated from methods in previous sections, can now arise from RHS. Therefore it would be expected this method engenders new relations when $w\geq 6$.
~\\[0.01in]

\subsection{Multiple integral}\label{MIsec}
The following method gives two new relations when weight $w$ is even, and one when $w$ is odd. It starts to take effect when $w\geq 7$. 
\begin{lemma} When $n\geq 1$ an integer,
$$\label{MI1}\int_0^1 {\int_0^1 {\frac{{{{\log }^{2n}}(\frac{{1 - y}}{{1 - x}})}}{{(1 + x)(1 + y)}}dxdy} }  = 2\int_0^1 {\frac{{{{\log }^{2n}}(1 - u)\log (1 + u)}}{u}du}  + \frac{2}{{2n + 1}}\int_0^1 {\frac{{{{\log }^{2n + 1}}(1 - u)}}{{1 + u}}du} $$
\end{lemma}
\begin{proof}
Start with the LHS $$I = \int_0^1 {\int_0^1 {\frac{{{{\log }^n}(\frac{{1 - y}}{{1 - x}})}}{{(1 + x)(1 + y)}}dxdy} } $$ replace $y$ by $u = \frac{{1 - y}}{{1 - x}}$, 
gives $$I = \int_0^1 {\int_0^{\frac{1}{{1 - x}}} {\frac{{{{\log }^n}u}}{{(1 + x)(2 - u + ux)}}(1 - x)dudx} } $$ now exchange order of integration (easily justified via Fubini's theorem), yields
$$\begin{aligned}
I &= \int_0^1 {\int_0^1 {\frac{{{{\log }^n}u}}{{(1 + x)(2 - u + ux)}}(1 - x)dxdu} }  + \int_1^\infty  {\int_{(u - 1)/u}^1 {\frac{{{{\log }^n}u}}{{(1 + x)(2 - u + ux)}}(1 - x)dxdu} } \\
 &= \int_0^1 {\frac{{{{\log }^n}u}}{{u(1 - u)}}\left[ {u\log 2 + \log (\frac{{2 - u}}{2})} \right]du}  + \int_1^\infty  {{{\log }^n}u\left( { - \frac{{\log 2}}{u} - \frac{{\log (2 - 1/u)}}{{1 - u}}} \right)du} \\
 &= \int_0^1 {\frac{{{{\log }^n}u}}{{u(1 - u)}}\left[ {u\log 2 + \log (\frac{{2 - u}}{2})} \right]du}  + {( - 1)^n}\int_0^1 {{{\log }^n}u\left( { - u\log 2 - \frac{{\log (2 - u)}}{{1 - 1/u}}} \right)\frac{1}{{{u^2}}}du}\end{aligned} $$
 Rewrite RHS in the form of polylogarithm integrals, when $n$ is odd, it vanishes (which can also be observed from the appearance of $I$ as double integral); when $n$ is even, it reduces to the form mentioned in the theorem.
\end{proof}

Dismantle the double integral by expanding $(\log(1-y)-\log(1-x))^{2n}$, by Lemma \ref{Gamma1}, we know that $i_{2n,0,1,1}$ can be always be expressed in terms of polylogarithm, for example:
\begin{multline*}
i_{6011} = \int_0^1 \frac{\log^6 (1-x)\log(1+x)}{x} dx = 630 \text{Li}_5\left(\frac{1}{2}\right) \zeta (3)-360 \text{Li}_4\left(\frac{1}{2}\right){}^2-60 \pi ^2 \text{Li}_6\left(\frac{1}{2}\right)+\\ 720 \text{Li}_8\left(\frac{1}{2}\right)+120 \text{Li}_5\left(\frac{1}{2}\right) \log ^3(2)+360 \text{Li}_6\left(\frac{1}{2}\right) \log ^2(2)-60 \pi ^2 \text{Li}_5\left(\frac{1}{2}\right) \log (2)+720 \text{Li}_7\left(\frac{1}{2}\right) \log (2)\end{multline*}

The above lemma only contributes when $w$ is even, the following additional one is valid regardless of parity. Although intimidating-looking, it is a valuable high-order relation that can be written explicitly for arbitrary $n$:
\begin{lemma} When $n\geq 1$ an integer,
\begin{multline}\label{MI2}
(n + 1)\int_0^1 {\int_0^1 {\frac{{{{\log }^n}(\frac{{1 - y}}{{1 + x}})\log (1 - x)}}{{(1 + x)(1 + y)}}dxdy} }  = \int_0^1 {\frac{{{{\log }^{n + 1}}(\frac{{1 + u}}{2})\log (1 - u)}}{u}du}  - (n + 1)\int_0^1 {\frac{{{{\log }^n}(\frac{{1 + u}}{2})\log \left( {\frac{{1 + u}}{{2u}}} \right)\log (1 - u)}}{{1 + u}}du} \\ + \frac{{n + 1}}{2}\int_0^1 {\frac{{{{\log }^2}(1 - u){{\log }^n}(\frac{{1 + u}}{2})}}{{1 + u}}du}  + \int_0^1 {\frac{{{{\log }^{n + 1}}(1 - u)}}{u}\log \left( {\frac{{1 + u}}{{2u}}} \right)du}  - \frac{1}{{n + 2}}\int_0^1 {\frac{{{{\log }^{n + 2}}(1 - u)}}{{u(u + 1)}}du}
\end{multline} Note that RHS, after expansion, can be cast into linear combination of standard polylogarithm integrals.
\end{lemma}
\begin{proof}
Replace $y$ by $u = \frac{{1 - y}}{{1 + x}}$, then exchange order of integration gives 
$$\int_0^1 {\int_0^1 {\frac{{{{\log }^n}(\frac{{1 - y}}{{1 + x}})\log (1 - x)}}{{(1 + x)(1 + y)}}dxdy} } = \int_0^{1/2} {\int_0^1 {\frac{{{{\log }^n}u\log (1 - x)}}{{2 - u - ux}}dxdu} }  + \int_{1/2}^1 {\int_0^{(1 - u)/u} {\frac{{{{\log }^n}u\log (1 - x)}}{{2 - u - ux}}dxdu} } $$
It is elementary to check that
$$\int_0^1 {\frac{{\log (1 - x)}}{{2 - u - ux}}dxdu}  = \frac{1}{u}\Li_2\left( {\frac{u}{{2(u - 1)}}} \right)$$ and
$$\int_0^{(1 - u)/u} {\frac{{{{\log }^n}u\log (1 - x)}}{{2 - u - ux}}dxdu}  = \frac{1}{u}\left( { - \log (2 - 2u)\log \frac{u}{{ - 1 + 2u}} - \Li_2\left( {\frac{{1 - 2u}}{{2\left( {1 - u} \right)}}} \right) + \Li_2\left( {\frac{u}{{2\left( { - 1 + u} \right)}}} \right)} \right)$$
therefore the double integral is
$$\int_0^{1/2} {\frac{{{{\log }^n}u}}{u}\Li_2\left( {\frac{u}{{2(u - 1)}}} \right)}  + \int_{1/2}^1 {\frac{{{{\log }^n}u}}{u}\left( { - \log (2 - 2u)\log \frac{u}{{ - 1 + 2u}} - \Li_2\left( {\frac{{1 - 2u}}{{2\left( {1 - u} \right)}}} \right) + \Li_2\left( {\frac{u}{{2\left( { - 1 + u} \right)}}} \right)} \right)} $$
this equals, via integration by parts:
$$ - \int_{1/2}^1 {\frac{{{{\log }^n}u}}{u}\log (2 - 2u)\log \frac{u}{{ - 1 + 2u}}}  + \frac{1}{{n + 1}}\int_0^1 {\frac{{{{\log }^{n + 1}}u}}{{u(1 - u)}}\log \frac{{2 - u}}{{2\left( {1 - u} \right)}}du}  - \frac{1}{{n + 1}}\int_{1/2}^1 {\frac{{\log \frac{1}{{2 - 2u}}}}{{\left( {1 - u} \right)\left( { - 1 + 2u} \right)}}{{\log }^{n + 1}}u} $$
the second integral is amenable into standard polylogarithm integral, while first and third are not so individually, but they miraculous combine into standard forms, giving the RHS of the lemma.
\end{proof}
When $w=7$, this relation makes $\Li_4(1/2)\zeta(3)$ to the party, which is not otherwise obtainable, therefore this relation is expected independent to previous ones.
~\\[0.01in]

\subsection{Hypergeometric $_2F_1$}
The following method provides new relation when $w\geq 9$, it introduces new combination of constants by multiplying them with $\zeta(n)$ for various $n$.

\begin{lemma}
The following holds when $(a,b,c)$ is near $(1,1,2)$:
\begin{multline}\label{2F1eq}
\sin \pi (b - a)\int_0^1 {{x^{b - 1}}{{(1 - x)}^{c - b - 1}}{{(1 + x)}^{ - a}}dx}  = \sin \pi (c - a)\int_0^1 {{x^{a - c}}\left[ {{{(1 - x)}^{c - b - 1}}{{(1 + x)}^{ - a}} - 1} \right]dx}  - \\ \sin \pi (b - a)\frac{{\Gamma (c - b)\Gamma (b)}}{{\Gamma (a)\Gamma (c - a)}}\int_0^1 {{x^{b - c}}\left[ {{{(1 - x)}^{c - a - 1}}{{(1 + x)}^{ - b}} - 1} \right]dx}  + \frac{{\sin \pi (c - a)}}{{a - c + 1}} - \frac{{\pi \Gamma (b)}}{{\Gamma (2 - c + b)\Gamma (a)\Gamma (c - a)}}\end{multline}
Note that each term is analytic near $(1,1,2)$.
\end{lemma}
\begin{proof}
Begin with the following identity, which relates three of Kummer's 24 solutions \cite{slater}, valid for $|\arg(-z)|<\pi$:
\begin{multline*}
\frac{{\Gamma (a)\Gamma (b)}}{{\Gamma (c)}}{_2F_1}(a,b;c;z) = \frac{{\Gamma (a)\Gamma (b - a)}}{{\Gamma (c - a)}}{( - z)^{ - a}}{_2F_1}(a,1 + a - c;1 + a - b;{z^{ - 1}}) \\ +  \frac{{\Gamma (b)\Gamma (a - b)}}{{\Gamma (c - b)}}{( - z)^{ - b}}{_2F_1}(b,1 + b - c;1 + b - a;{z^{ - 1}})\end{multline*}
Plug in $z=-1$ and the integral representation of ${_2F_1}$ gives the result.
\end{proof}

To obtain a relation between weight $w$ integrals, we apply $$\frac{\partial }{{\partial {a^m}\partial {b^n}\partial {c^r}}}$$ on both sides of (\ref{2F1eq}), with $m+n+r = w$, then plug in $a=1,b=1,c=2$. By varying $m,n,r$, we could obtain plenty of new relations when $w\geq 9$.\par 

It is an empirical curiosity that, $w$ is odd, this method subsumed both contour integration (Section \ref{CIsec}) and multiple integral (Section \ref{MIsec}); but not so when $w$ is even.

~\\[0.02in]
\section{Explicit evaluation}
\begin{table}[t]
\captionsetup{font=scriptsize}
\centering
\begin{tabular}{l|lllllllc}
\cline{2-9}
\multicolumn{1}{c|}{Weight $w$} & \multicolumn{1}{c|}{IBP} & \multicolumn{1}{c|}{FT} & \multicolumn{1}{c|}{GP} & \multicolumn{1}{c|}{SR} & \multicolumn{1}{c|}{CI} & \multicolumn{1}{c|}{MI} & \multicolumn{1}{c|}{2F1} & \multicolumn{1}{c|}{Total $ (3w^2+w-2)/2$} \\ \hline
\multicolumn{1}{|l|}{3} & 6 & 10 & 14 & - & - & - & - & 14 \\ \cline{1-1}
\multicolumn{1}{|l|}{4} & 10 & 17 & 22 & 25 & - & - & - & 25 \\ \cline{1-1}
\multicolumn{1}{|l|}{5} & 15 & 27 & 33 & 39 & - & - & - & 39 \\ \cline{1-1}
\multicolumn{1}{|l|}{6} & 21 & 38 & 45 & 53 & 55 & 55 & 55 & 56 \\ \cline{1-1}
\multicolumn{1}{|l|}{7} & 28 & 52 & 60 & 70 & 73 & 74 & 74 & 76 \\ \cline{1-1}
\multicolumn{1}{|l|}{8} & 36 & 67 & 76 & 88 & 94 & 96 & 96 & 99 \\ \cline{1-1}
\multicolumn{1}{|l|}{9} & 45 & 85 & 95 & 109 & 118 & 119 & 120 & 125 \\ \cline{1-1}
\multicolumn{1}{|l|}{10} & 55 & 104 & 115 & 131 & 144 & 146 & 148 & 154 \\ \cline{1-1}
\multicolumn{1}{|l|}{11} & 66 & 126 & 138 & 156 & 172 & 173 & 177 & 186 \\ \cline{1-1}
\multicolumn{1}{|l|}{12} & 78 & 149 & 162 & 182 & 203 & 205 & 210 & 221 \\ \cline{1-1}
\multicolumn{1}{|l|}{13} & 91 & 175 & 189 & 211 & 237 & 238 & 245 & 259 \\ \cline{1-1}
\multicolumn{1}{|l|}{14} & 105 & 202 & 217 & 241 & 273 & 275 & 284 & 300 \\ \cline{1-1}
\multicolumn{1}{|l|}{15} & 120 & 232 & 248 & 274 & 311 & 312 & 324 & 344 \\ \cline{1-1}
\multicolumn{1}{|l|}{16} & 136 & 263 & 280 & 308 & 352 & 354 & 368 & 391 \\ \cline{1-1}
\multicolumn{1}{|l|}{17} & 153 & 297 & 315 & 345 & 396 & 397 & 414 & 441 \\ \cline{1-1}
\multicolumn{1}{|l|}{18} & 171 & 332 & 351 & 383 & 442 & 444 & 464 & 494 \\ \cline{1-1}
\multicolumn{1}{|l|}{19} & 190 & 370 & 390 & 424 & 490 & 491 & 515 & 550 \\ \cline{1-1}
\multicolumn{1}{|l|}{20} & 210 & 409 & 430 & 466 & 541 & 543 & 570 & 609 \\ \cline{1-1}
\end{tabular}
\caption{Number of independent relations obtained via different methods when they are adjoined \textit{successively}. (IBP: integration by parts; FT: fractional transformation; GP: gamma and polylogarithm; SR: square replacement; CI: contour integration; MI: multiple integral; 2F1: hypergeometric $_2F_1$; Total: number of polylogarithmic integrals of that weight)}
\label{table 1}
\end{table}

\subsection{Weight 3, 4 and 5}
We start from values of six weight $2$ integrals:
$$i_{0100} = -\frac{\pi^2}{6} \qquad i_{0011} = \frac{\pi^2}{12} \qquad i_{1001} = -\frac{\pi^2}{6}$$
$$i_{1002} = -\frac{\pi^2}{12} + \frac{\log^2 2}{2}\qquad i_{0012} = \frac{\log^2 2}{2} \qquad i_{0102} = -\frac{\pi^2}{12}$$
where we used the value of $\Li_2(1/2)$, and start deducing higher order polylogarithm integrals. Table \ref{table 1} indicates that all $14$ weight-$3$ integrals, $25$ weight-$4$ integrals and $39$ weight-$5$ integrals can be expressed in terms of Riemann zeta and $\Li_n(1/2)$. For example, $$\int_0^1 \frac{\log(1-x)\log(1+x)}{1+x}dx -\frac{1}{12} \pi^2 \log (2)+\frac{\zeta (3)}{8}+\frac{\log ^3(2)}{3}$$
$$\int_0^1 \frac{\log(1-x)\log x \log(1+x)}{x} dx = 2 \text{Li}_4\left(\frac{1}{2}\right)+\frac{7}{4} \zeta (3) \log (2)-\frac{3 \pi ^4}{160}+\frac{\log ^4(2)}{12}-\frac{1}{12} \pi ^2 \log ^2(2)$$
\begin{multline*}
\int_0^1 \frac{\log(1-x) \log^3(1+x)}{x} dx =-6 \text{Li}_5\left(\frac{1}{2}\right)-6 \text{Li}_4\left(\frac{1}{2}\right) \log (2)+\\ \frac{7 \pi ^2 \zeta (3)}{16}+\frac{3 \zeta (5)}{4}-\frac{21}{8} \zeta (3) \log ^2(2)-\frac{1}{5} \log ^5(2)+\frac{1}{6} \pi ^2 \log ^3(2)
\end{multline*}

The evaluation of such integrals of weight $\leq 5$ is well-known in literature, in an one-by-one manner. It might be pleasing to learn that these integrals can be calculated simultaneously without appealing to Euler sums or MZVs. 

\subsection{Weight 6}
This weight encounters the first constant which is probably not expressible using ordinary polylogarithm. Table \ref{table 1} shows one additional constant is needed, I choose it to be
$$\mathbf{F}_1 := \int_0^1 \frac{\log^4 x \log(1+x)}{1+x} dx = 24 \zeta(\bar{5},1) = 24\sum_{n=2}^\infty \frac{(-1)^n}{n^5}H_{n-1}\approx 0.633579571034807$$
All $56$ integrals can be expressed, for instance,
\begin{multline*}\int_0^1 \frac{\log x \log^4(1+x)}{1-x} dx = \frac{\textbf{F}_1}{2}+2 \pi ^2 \text{Li}_4\left(\frac{1}{2}\right)+48 \text{Li}_6\left(\frac{1}{2}\right)-12 \text{Li}_4\left(\frac{1}{2}\right) \log ^2(2)-\\ 24 \text{Li}_5\left(\frac{1}{2}\right) \log (2)-\frac{93 \zeta (3)^2}{16}+\frac{21}{2} \zeta (3) \log ^3(2)+93 \zeta (5) \log (2)- \\ \frac{9 \pi ^6}{140}-\frac{1}{30} 7 \log ^6(2)-\frac{1}{6} \pi ^2 \log ^4(2)-\frac{61}{120} \pi ^4 \log ^2(2)
\end{multline*}
\begin{multline*}
\int_0^1 \frac{\log(1-x) \log^2 x \log^2(1+x)}{1-x} dx = -\frac{4}{3} \pi ^2 \text{Li}_4\left(\frac{1}{2}\right)+\frac{131 \zeta (3)^2}{16}+\frac{7}{3} \zeta (3) \log ^3(2)- \frac{77}{24} \pi ^2 \zeta (3) \log (2)+\\ \frac{217}{8} \zeta (5) \log (2)+\frac{41 \pi ^6}{30240}-\frac{1}{18} \pi ^2 \log ^4(2)-\frac{1}{144} \pi ^4 \log ^2(2) 
\end{multline*}
\begin{multline*}
\int_0^1 \frac{\log^2 x \log^3(1+x)}{1+x} dx = \frac{\textbf{F}_1}{4}+24 \text{Li}_6\left(\frac{1}{2}\right)+12 \text{Li}_4\left(\frac{1}{2}\right) \log ^2(2)+ 24 \text{Li}_5\left(\frac{1}{2}\right) \log (2)-\\ 3 \zeta (3)^2+\frac{7}{2} \zeta (3) \log ^3(2)-\frac{53 \pi ^6}{2520}+\frac{\log ^6(2)}{3}-\frac{1}{4} \pi ^2 \log ^4(2)
\end{multline*}

The middle one being lacking $\mathbf{F}_1$ is curious, a more striking example occurs in weight $8$.\\

In general, $$\int_0^1 \frac{\log^k x \log(1+x)}{1+x} dx = k! \zeta(\overline{k+1},1) = k!\sum_{n=2}^\infty \frac{(-1)^n}{n^{k+1}}H_{n-1}$$ seems to define a new constant for even $k$, while for odd $k$, the corresponding Euler sum reduces to a combination of ordinary zeta values.\cite{Flajolet}.

\subsection{Weight 7}
There are two constants according to Table \ref{table 1}, I chose them to be
$$\begin{aligned}\mathbf{G}_1 &= \int_0^1 \frac{\log^4 x \log^2(1+x)}{1+x} dx \\ \mathbf{G}_2 &= \int_0^1 \frac{\log(1-x)\log^4 x\log(1+x)}{1-x} dx \end{aligned}$$

Some examples are:
\begin{multline*}
\int_0^1 \frac{\log(1-x) \log^4 x \log (1+x)}{1+x} dx = \mathbf{F}_1 \log (2)-\mathbf{G}_1+48 \text{Li}_4\left(\frac{1}{2}\right) \zeta (3)-\frac{85 \pi ^4 \zeta (3)}{48}-\\ \frac{39 \pi ^2 \zeta (5)}{2}+\frac{5655 \zeta (7)}{16}+2 \zeta (3) \log ^4(2)-2 \pi ^2 \zeta (3) \log ^2(2)+\frac{93}{4} \zeta (5) \log ^2(2)+ \frac{195}{4} \zeta (3)^2 \log (2)-\frac{37}{840} \pi ^6 \log (2)
\end{multline*}
\begin{multline*}
\int_0^1 \frac{\log^2(1-x) \log x \log^3 (1+x)}{1+x} dx = -\frac{1}{2} \mathbf{F}_1 \log (2)+\frac{3 \textbf{G}_1}{8}-\frac{\textbf{G}_2}{4}+12 \text{Li}_4\left(\frac{1}{2}\right) \zeta (3)-72 \text{Li}_7\left(\frac{1}{2}\right)-\\ \pi ^2 \text{Li}_4\left(\frac{1}{2}\right) \log (2)-\frac{103 \pi ^4 \zeta (3)}{480}-\frac{405 \pi ^2 \zeta (5)}{64}+\frac{2385 \zeta (7)}{16}+\frac{23}{4} \zeta (3) \log ^4(2)-\frac{67}{16} \pi ^2 \zeta (3) \log ^2(2)+\\ \frac{1209}{16} \zeta (5) \log ^2(2)+\frac{1005}{32} \zeta (3)^2 \log (2)+\frac{\log ^7(2)}{70}-\frac{23}{120} \pi ^2 \log ^5(2)-\frac{29}{240} \pi ^4 \log ^3(2)-\frac{1069 \pi ^6 \log (2)}{10080}
\end{multline*}

Note that the only relation that introduces constant $\Li_4 (1/2) \zeta(3)$ on the RHS is Lemma \ref{MI2}. There seems no non-trivial nice evaluation available for odd weights.

\subsection{Weight 8}
The three constants I shall introduce are $$\begin{aligned}\mathbf{H}_1 &= \int_0^1 \frac{\log^6 x \log(1+x)}{1+x} dx \\ \mathbf{H}_2 &= \int_0^1 \frac{\log x \log^3 x\log^4(1+x)}{x} dx \\ \mathbf{H}_3 &= \int_0^1 \frac{\log^2 (1-x) \log^4 x\log (1+x)}{1+x} dx \end{aligned}$$

One remarkable evaluation, mentioned in at the beginning of the article:
\begin{theo}\label{remarkableweight8}
\begin{multline*}
\int_0^1 \frac{\log^2 (1-x) \log^2 x \log^3(1+x)}{x} dx = -168 \text{Li}_5\left(\frac{1}{2}\right) \zeta (3)+96 \text{Li}_4\left(\frac{1}{2}\right){}^2-\frac{19}{15} \pi ^4 \text{Li}_4\left(\frac{1}{2}\right)+\\ 12 \pi ^2 \text{Li}_6\left(\frac{1}{2}\right)+8 \text{Li}_4\left(\frac{1}{2}\right) \log ^4(2)-2 \pi ^2 \text{Li}_4\left(\frac{1}{2}\right) \log ^2(2)+12 \pi ^2 \text{Li}_5\left(\frac{1}{2}\right) \log (2)+\frac{87 \pi ^2 \zeta (3)^2}{16}+\\ \frac{447 \zeta (3) \zeta (5)}{16}+\frac{7}{5} \zeta (3) \log ^5(2)-\frac{7}{12} \pi ^2 \zeta (3) \log ^3(2)-\frac{133}{120} \pi ^4 \zeta (3) \log (2)-\frac{\pi ^8}{9600}+\frac{\log ^8(2)}{6}- \\ \frac{1}{6} \pi ^2 \log ^6(2)-\frac{1}{90} \pi ^4 \log ^4(2)+\frac{19}{360} \pi ^6 \log ^2(2)
\end{multline*}
\end{theo}
This explicit result can still be obtained if one forgoes Lemma \ref{MI2}. 

Another surprising result is 
\begin{multline}\label{conj1w8}
\int_0^1 \frac{\log^4 (1-x) \log^3(1+x)}{1+x} dx = -12 \pi ^2 \zeta (3)^2+288 \zeta (3) \zeta (5)+12 \zeta (3) \log ^5(2)-12 \pi ^2 \zeta (3) \log ^3(2)+\\ 168 \zeta (5) \log ^3(2)+108 \zeta (3)^2 \log ^2(2)-2 \pi ^4 \zeta (3) \log (2)-48 \pi ^2 \zeta (5) \log (2)+\\ 720 \zeta (7) \log (2)- \frac{499 \pi ^8}{25200}+\frac{\log ^8(2)}{8}-\frac{1}{3} \pi ^2 \log ^6(2)-\frac{19}{60} \pi ^4 \log ^4(2)-\frac{1}{6} \pi ^6 \log ^2(2)
\end{multline}

It becomes more and more lengthy to write down a general explicit evaluation. For instance
\begin{multline*}\int_0^1 \frac{\log^3 (1-x) \log^4 x}{1+x} dx  = \frac{5 \pi ^2 \textbf{F}_1}{2}-3 \textbf{F}_1 \log ^2(2)+3 \textbf{G}_1 \log (2) -\frac{\textbf{H}_1}{2}+\textbf{H}_2+ \\ 288 \text{Li}_5\left(\frac{1}{2}\right) \zeta (3)+\frac{28}{5} \pi ^4 \text{Li}_4\left(\frac{1}{2}\right)-63 \pi ^2 \zeta (3)^2+ \frac{2259 \zeta (3) \zeta (5)}{2}-\frac{1}{5} 12 \zeta (3) \log ^5(2)+ \\4 \pi ^2 \zeta (3) \log ^3(2)+\frac{93}{2} \zeta (5) \log ^3(2)-36 \zeta (3)^2 \log ^2(2)+\frac{22}{5} \pi ^4 \zeta (3) \log (2)-\\ \frac{369}{4} \pi ^2 \zeta (5) \log (2)+\frac{11475}{8} \zeta (7) \log (2)-\frac{523 \pi ^8}{3360}+\frac{7}{30} \pi ^4 \log ^4(2)-\frac{10}{21} \pi ^6 \log ^2(2)
\end{multline*}

\subsection{Weight 9 and 10}
The full results for all weight 9 and weight 10 integrals are available in  \textit{Mathematica} format (Section \ref{mathematica}). Here I extract some notable ones. \\
Firstly, we have a nice two-term result, without any 'new' weight $6,7$ or $8$ constants.
\begin{multline*}
i_{3052} = -\frac{i_{6022}}{2}+120 \zeta (3)^3-\frac{4 \pi ^6 \zeta (3)}{3}-18 \pi ^4 \zeta (5)-360 \pi ^2 \zeta (7)+6720 \zeta (9)+ \\ 21 \zeta (3) \log ^6(2)-25 \pi ^2 \zeta (3) \log ^4(2)+480 \zeta (5) \log ^4(2)+300 \zeta (3)^2 \log ^3(2)-13 \pi ^4 \zeta (3) \log ^2(2)-\\ 240 \pi ^2 \zeta (5) \log ^2(2)+ 4320 \zeta (7) \log ^2(2)-60 \pi ^2 \zeta (3)^2 \log (2)+2880 \zeta (3) \zeta (5) \log (2)+\frac{\log ^9(2)}{6}-\\ \frac{1}{2} \pi ^2 \log ^7(2)-\frac{49}{60} \pi ^4 \log ^5(2)- \frac{107}{126} \pi ^6 \log ^3(2)-\frac{13}{45} \pi ^8 \log (2)
\end{multline*}

The family $i_{a0b2}$ seems to possess many nice properties, for example, this weight $10$ example:
\begin{multline}\label{conj1w10}\int_0^1 \frac{\log^5(1-x)\log^4 (1+x)}{1+x} dx = -20 \pi ^4 \zeta (3)^2+7200 \zeta (5)^2-960 \pi ^2 \zeta (3) \zeta (5)+14400 \zeta (3) \zeta (7)+20 \zeta (3) \log ^7(2)-\\ 40 \pi ^2 \zeta (3) \log ^5(2)+600 \zeta (5) \log ^5(2)+600 \zeta (3)^2 \log ^4(2)-\frac{76}{3} \pi ^4 \zeta (3) \log ^3(2)- 560 \pi ^2 \zeta (5) \log ^3(2)+\\ 8640 \zeta (7) \log ^3(2)-360 \pi ^2 \zeta (3)^2 \log ^2(2)+ 10080 \zeta (3) \zeta (5) \log ^2(2)+1440 \zeta (3)^3 \log (2)-\frac{20}{3} \pi ^6 \zeta (3) \log (2)-\\ 112 \pi ^4 \zeta (5) \log (2)- 2400 \pi ^2 \zeta (7) \log (2)+  40320 \zeta (9) \log (2)- \frac{149 \pi ^{10}}{1320}+\frac{\log ^{10}(2)}{10}-\frac{5}{12} \pi ^2 \\ \log ^8(2)-\frac{7}{9} \pi ^4 \log ^6(2)- \frac{19}{18} \pi ^6 \log ^4(2)-\frac{47}{60} \pi ^8 \log ^2(2)\end{multline}
A conjecture generalizing this is \ref{conj1}.

Apart from the above, the only non-trivial weight $10$ explicit result seems to be
\begin{multline}\label{conj2w10}\int_0^1 \frac{\log(1-x)\log^6 x \log^2(1+x)}{x} = \frac{5 \pi ^4 \textbf{F}_1}{4}+\frac{\pi ^2 \textbf{H}_1}{2}+6 \pi ^6 \text{Li}_4\left(\frac{1}{2}\right)-\frac{105 \pi ^4 \zeta (3)^2}{16}- \frac{945}{8} \pi ^2 \zeta (3) \zeta (5)-\\ \frac{11475 \zeta (3) \zeta (7)}{2}-\frac{11475 \zeta (5)^2}{4}+\frac{21}{4} \pi ^6 \zeta (3) \log (2)+\frac{1633 \pi ^{10}}{24640}+\frac{1}{4} \pi ^6 \log ^4(2)-\frac{1}{4} \pi ^8 \log ^2(2)\end{multline}
see also \ref{conj2}. \\

It is should be noted that in the pool of weight $9$ constants, $\Li_4(1/2) \Li_5(1/2)$ is  lacking. While for weight $10$, despite appearance of both $\Li_5(1/2)^2$ and $\Li_4(1/2)\Li_6(1/2)$, they are introduced by only one relation (Lemma \ref{MI2}). Therefore it is reasonable to suspect there might still be higher weight sporadic relation that introduce these constants.

\section{Discussions}
\subsection{Dimension of $\mathcal{I}_w / \mathcal{I}_{w-1}$}\label{residualspace}
\quad We first recapitulate some definitions mentioned in the introduction. $\mathcal{A}$ is the algebra generated over $\mathbb{Q}$ by all polylogarithm integrals, and $\mathcal{I}_n \subset \mathcal{A}$ is the ideal by those with weight $\leq n$. \par 
Let $N_w$ the number of "primitive" constant of weight $w$, for example, $\zeta(3), \Li_n(1/2), \textbf{F}_1$ for $n\geq 4$ is primitive, while $\Li_4(1/2)\Li_7(1/2)$ is not a primitive weight $11$ constant. Although being a vague concept, we only use it to illustrate one point: the number of possible product combinations of weight $w$ constants is the coefficient of $x^w$ in
$$\prod_{w=1}^\infty (1-x)^{-N_w}$$ 
It is expected that $N_w \geq 1$, so number of total combinations of weight $w$ is $\geq p(w)$, the number of partition of an integer, so grows like $e^{\sqrt{w}}$. Therefore when $w$ is large, the number of constants exceeds the number of integrals. A similar analysis on Euler sums has been done in \cite{Flajolet}. \par

Therefore in higher weight, it would be more conducive forget these constants and concentrate on existence of relation themselves, that is, the quantity $\dim_\mathbb{Q} \mathcal{I}_w / \mathcal{I}_{w-1}$.

\begin{table}[h]
\captionsetup{font=scriptsize}
\centering
\begin{tabular}{cccccccccc}
$w$                                                                        & 3                       & 4                       & 5                       & 6                       & 7                       & 8                       & 9                       & 10                      & 11                      \\ \cline{2-10} 
\multicolumn{1}{c|}{$\dim_\mathbb{Q} \mathcal{I}_w / \mathcal{I}_{w-1}$} & \multicolumn{1}{c|}{1}  & \multicolumn{1}{c|}{1}  & \multicolumn{1}{c|}{2}  & \multicolumn{1}{c|}{2}  & \multicolumn{1}{c|}{4}  & \multicolumn{1}{c|}{4}  & \multicolumn{1}{c|}{7}  & \multicolumn{1}{c|}{7}  & \multicolumn{1}{c|}{11} \\ \hline
$w$                                                                        & 12                      & 13                      & 14                      & 15                      & 16                      & 17                      & 18                      & 19                      & 20                      \\ \cline{2-10} 
\multicolumn{1}{c|}{$\dim_\mathbb{Q} \mathcal{I}_w / \mathcal{I}_{w-1}$} & \multicolumn{1}{c|}{12} & \multicolumn{1}{c|}{16} & \multicolumn{1}{c|}{17} & \multicolumn{1}{c|}{22} & \multicolumn{1}{c|}{24} & \multicolumn{1}{c|}{29} & \multicolumn{1}{c|}{31} & \multicolumn{1}{c|}{37} & \multicolumn{1}{c|}{40} \\ \hline
\end{tabular}
\caption{\textbf{Upper bound} for dimension of "new constants" for each weight $\leq 20$}
\label{tab:my-table}
\end{table}

For $w\geq 4$, the above table can deduced from Table 1 by calculating the difference between last two columns, then add $1$ if $w$ is even and $2$ if $w$ is odd, corresponding to $\Li_w(1/2)$ and $\zeta(w)$. My humble feeling is that these bounds can still be considerably lowered, especially when $w\geq 11$, due to not-yet-discovered methods to generate relations.

\subsection{\textit{Mathematica} files}\label{mathematica}
Results of weight $\leq 12$ have been fully computed, and is available in Mathematica format \cite{ESIntegrate}. New constants of weight $6-8$ are named as in previous section, and those of weight $9-12$ are not named. To compute, for example, $i_{3442}$, input
~\\[0.02in]

{\fontfamily{qcr}\selectfont
ESIntegrate[Log[1-x]\textasciicircum 3 Log[x]\textasciicircum 4 Log[1+x]\textasciicircum 4/ (1+x),\{x,0,1\}]}

~\\[0.01in]
For $13\leq w\leq 20$, the same file contains results modulo $\mathcal{I}_{w-1}\oplus \mathbb{Q}\Li_w(1/2)$. For example, the entry 
~\\[0.01in]

{\fontfamily{qcr}\selectfont
i4452, -((3 i8322)/7) - (3 i8412)/14 - i9222/7 + (2 i9312)/63 + (
  2 i9402)/63 + (10 iA212)/21 + (8 iA302)/45 - (20 iB112)/231 - (
  296 iB202)/1155 + (97 iC102)/2310}
~\\[0.01in] 
implies\footnote{Here $A=10, B=11$ and so on}
$$i_{4452}-\left(-\frac{3 i_{8322}}{7}-\frac{3i_{8412}}{14}-\frac{i_{9222}}{7}+\frac{2 i_{9312}}{63}+\frac{2 i_{9402}}{63}+\frac{10 i_{A212}}{21}+\frac{8 i_{A302}}{45}-\frac{20 i_{B112}}{231}-\frac{296 i_{B202}}{1155}+\frac{97 i_{C102}}{2310}\right)$$
is in $\mathcal{I}_{13}\oplus \mathbb{Q}\Li_{14}(1/2)$.

\subsection{Some conjectures}
\begin{conjecture}\label{conj1}
For any $n\geq 1$, the value of
$$\int_0^1 \frac{\log^{n+1}(1-x) \log^n(1+x)}{1+x} dx$$
is in the algebra over $\mathbb{Q}$ generated by $\pi^2, \log 2$ and $\zeta(n)$ for $n$ odd.
\end{conjecture}
The case for $n=2, w=6$ is
\begin{multline*}\int_0^1 \frac{\log^3(1-x) \log^2(1+x)}{1+x} dx = 6 \zeta (3)^2+6 \zeta (3) \log ^3(2)-2 \pi ^2 \zeta (3) \log (2)+24 \zeta (5) \log (2)-\\ \frac{23 \pi ^6}{2520}+\frac{\log ^6(2)}{6}-\frac{1}{4} \pi ^2 \log ^4(2)-\frac{1}{12} \pi ^4 \log ^2(2)\end{multline*}
The case for $n=3,4$ are (\ref{conj1w8}) and (\ref{conj1w10}) respectively. The weight $12$ case for $n=5$ is also true, look into the \textit{Mathematica} files for the value. \par

The value of such integral should still be derivable even if relations from ${_2F_1}$ method are omitted, but a unified proof for general $n$ remains elusive to me.

\begin{conjecture}\label{conj2}
For any $n\geq 0$, the value of
$$\int_0^1 \frac{\log(1-x) \log^{2n} x \log^2 (1+x) }{x} dx \in \mathcal{I}_{2n+2}$$
\end{conjecture}
The case $n=0$ is the well-known
$$\int_0^1 \frac{\log(1-x) \log^2 (1+x) }{x} dx = -\frac{\pi^4}{240}$$
For $n=1$, result of integral is:
$$2 \pi ^2 \text{Li}_4\left(\frac{1}{2}\right)-\frac{15 \zeta (3)^2}{4}+\frac{7}{4} \pi ^2 \zeta (3) \log (2)-\frac{163 \pi ^6}{10080}+\frac{1}{12} \pi ^2 \log ^4(2)-\frac{1}{12} \pi ^4 \log ^2(2)$$
$n=3$ is the equation (\ref{conj2w10}), again see the \textit{Mathematica} files for the case $n=2,4$. 
\par 
Both conjectures are consistent with the behaviour of $i_{1,2n,2,1}$ modulo constants, complied for weight $\leq 20$. (See Section \ref{residualspace}). \\
An expression for $i_{a0c2}$ in terms of MZVs and multiple polylogarithm is given by Xu at \cite{XuCe}, but it seems not directly applicable for our problem.
~\\[0.03in]


\begin{thebibliography}{11}
\bibitem{2}
Coffey, M. W. (2010). Some definite logarithmic integrals from Euler sums, and other integration results. \textit{arXiv preprint arXiv}:1001.1366. https://arxiv.org/abs/1001.1366

\bibitem{Flajolet}
Flajolet, P., Salvy, B. (1998). Euler sums and contour integral representations. \textit{Experimental Mathematics}, 7(1), 15-35.

\bibitem{1}
Laurenzi, B. J. (2010). Logarithmic Integrals, Polylogarithmic Integrals and Euler sums. \textit{arXiv preprint arXiv}:1010.6229. https://arxiv.org/pdf/1010.6229.pdf

\bibitem{slater}
Slater, L. (1966). \textit{Generalized hypergeometric functions}. Cambridge University Press.

\bibitem{XuCe}
Xu, C. (2019). Integrals of logarithmic functions and alternating multiple zeta values. \textit{Mathematica Slovaca}, 69(2), 339-356.

\bibitem{3}
Xu, C. (2017). Evaluations of Euler-Type Sums of Weight $\le 5$. \textit{Bulletin of the Malaysian Mathematical Sciences Society}, 1-31.

\bibitem{ESIntegrate}
Au, K. C. ESIntegrate 1-12 upload. doi: 10.13140/RG.2.2.21310.43840

\end{thebibliography}
\end{document}